\documentclass{amsart}

\usepackage[nobysame]{amsrefs}
\usepackage{amsthm}
\usepackage{amssymb}

\usepackage{tikz-cd}
\usepackage{colonequals}
\usepackage{enumerate}
\usepackage{eucal}

\usepackage{hyperref}
\hypersetup{%
  bookmarksnumbered=true,%
  colorlinks=true,%
  linkcolor=blue,%
  citecolor=blue,%
  filecolor=blue,%
  menucolor=blue,%
  urlcolor=blue,%
  bookmarksopen=true,%
  bookmarksdepth=2,%
  pageanchor=true}

\numberwithin{equation}{section}

\swapnumbers

\theoremstyle{plain}
\newtheorem{theorem}[equation]{Theorem}
\newtheorem{proposition}[equation]{Proposition}
\newtheorem{lemma}[equation]{Lemma}

\theoremstyle{definition}

\newtheorem{example}[equation]{Example}

\newtheorem{chunk}[equation]{}

\theoremstyle{remark}

\newcommand{\ann}{\operatorname{ann}}
\newcommand{\aqh}[4]{\operatorname{D}_{#1}({#2}/{#3};#4)}

\newcommand{\depth}{\operatorname{depth}}
\renewcommand{\dim}{\operatorname{dim}}

\newcommand{\edim}{\operatorname{edim}}

\newcommand{\fdim}{\operatorname{flat\,dim}}
\newcommand{\fm}{\mathfrak{m}} 
\newcommand{\hh}{\operatorname{H}}
\newcommand{\Ker}{\operatorname{Ker}}
\newcommand{\pdim}{\operatorname{proj\,dim}}

\newcommand{\Tor}{\operatorname{Tor}}

\newcommand{\vf}{\varphi}

\newcommand{\lra}{\longrightarrow}
\newcommand{\xra}{\xrightarrow}

\begin{document}

\title[Freeness criterion]{A freeness criterion without patching\\ for modules over local rings }

\author[S.~Brochard]{Sylvain Brochard}
\address{Institut Montpell\'ierain Alexander Grothendieck, CNRS, Univ. Montpellier}
\email{sylvain.brochard@umontpellier.fr}

\author[S.~B.~Iyengar]{Srikanth B.~Iyengar}
\address{Department of Mathematics,
University of Utah, Salt Lake City, UT 84112, U.S.A.}
\email{iyengar@math.utah.edu}

\author[C.~B.~Khare]{Chandrashekhar  B. Khare}
\address{Department of Mathematics,
University of California, Los Angeles, CA 90095, U.S.A.}
\email{shekhar@math.ucla.edu}

\thanks{Partly supported by NSF grant DMS-200985 (SBI)}

\date{\today}


\begin{abstract} 
It is proved that if $\varphi\colon A\to B$ is a local homomorphism of commutative noetherian local rings, a nonzero finitely generated $B$-module $N$ whose flat dimension over $A$ is at most $\edim A - \edim B$, is free over $B$, and $\vf$ is a special type of complete intersection. This result is motivated by a ``patching method" developed by Taylor and Wiles, and a conjecture of de Smit, proved by the first author, dealing with the special case when $N$ is flat over $A$.
\end{abstract}

\maketitle
   
\section{Introduction}

The work of Wiles and Taylor-Wiles \cite{Taylor/Wiles:1995,Wiles:1995} on modularity lifting theorems relies on a  patching  method that has been generalized to prove a series of remarkable results that verify in many cases  the Fontaine-Mazur conjecture. The Fontaine-Mazur conjecture relates {\it geometric}  representations of absolute Galois groups of number fields to automorphic forms and motives. The spectacular proof  by Newton and Thorne  \cite{Newton/Thorne:2020} of the automorphy of all symmetric powers of Galois representations  associated to classical newforms is  a recent example.

One of the ingredients in the patching method of Wiles, Taylor-Wiles  et.\ al.\  is a commutative algebra result  that for our purposes is best summarized in \cite[Proposition 2.1]{Diamond:1997}.   The  assumptions there are used to reduce the  proof to a  statement about modules over regular local rings, which follows from the Auslander-Buchsbaum formula.  We spell this  out in the  statement below, as  it  gives  a context for the results  we prove here.

\begin{proposition}
\label{pr:AB}
Let $\mathcal O$ be a DVR, $A\colonequals \mathcal O[\![x_1,\ldots, x_r]\!]$, and $\vf\colon A \to  B$ a  map  of $\mathcal O$-algebras. Suppose the ring $B$ is a quotient of $\mathcal O[\![y_1,\ldots, y_s]\!]$ and $N$ is  a nonzero $B$-module that is finitely generated as an $A$-module and satisfies  $\pdim_AN \leq r-s$. Then $N$ is a free $B$-module and $B=\mathcal O[\![y_1,\ldots, y_s]\!]$. 
\end{proposition}

\begin{proof}
Since $N$ is finitely generated  over $A$, the depth formula of Auslander and Buchsbaum yields
 \[ 
 \depth_AN+\pdim_AN= \depth A = r+1
 \]
so that $\depth_AN \geq s+1$. This explains the inequality on the left:
\[
s+1\leq \depth_AN = \depth_B N \leq \dim B \leq \dim{\mathcal O}[\![y_1,\ldots, y_s]\!] = s+1\,.
\]
The equality holds as $N$ is finitely generated as an $A$-module, and the inequality on the right holds because $B$ is a quotient of $\mathcal O[\![y_1,\ldots, y_s]\!]$. It follows that $B$ and  $\mathcal O[\![y_1,\ldots, y_s]\!]$ have the same dimension, and hence that they are equal, for the former is a quotient of the latter, which is a domain. In particular $B$ itself is regular and $\pdim_BN$ is finite. Then given $\depth_BN=s+1$, another application of the Auslander-Buchsbaum formula yields that $N$ is a free $B$-module.
\end{proof}

Bart de Smit made the remarkable conjecture that   if $A\to B$ is a local homomorphism of (commutative) artinian local rings of the same embedding dimension, then any $B$-module that is flat as an $A$-module is also  flat as a $B$-module. This strengthens Proposition~\ref{pr:AB} in the case $r=s$ and allows one in principle to dispense with patching in the techniques \`a  la Wiles to prove modularity lifting theorems.

In \cite[Theorem 1.1]{Brochard:2017} the first author proved de Smit's conjecture. The connection to patching  is explained, for example, in \cite[\S 3]{Brochard:2017}.  In fact, in loc. cit.  it is proved  that if  $A,B$ are noetherian local rings with $\edim A\ge \edim B$  and $N$ is a  finitely generated $B$-module that is  flat over $A$,  then $N$ is flat over $B$.   We extend   \cite[Theorem 1.1]{Brochard:2017},  giving \emph{en passant} a new, and simpler, proof of  it, by  proving:

\begin{theorem}[see Theorem \ref{th:Brochard}]
\label{th:intro}
Suppose  $\vf\colon A\to B$ is a local homomorphism of noetherian local rings, and $N$ is a nonzero finitely generated $B$-module whose flat dimension over $A$ satisfies $\fdim_AN\le \edim A-\edim B$, then $N$ is free as a $B$ module and $\vf$ is an exceptional  complete intersection map. 
\end{theorem}

We say $\vf$ is exceptional complete intersection if it is complete intersection and $\edim A - \dim A= \edim B -  \dim B$; see ~\ref{ch:eci}. When $\vf$ is surjective, with kernel $I$, this property is equivalent to the condition that $I$ can be generated by a regular sequence whose images are linearly independent in the cotangent space $\fm_A/\fm_A^2$. There is a similar characterization of this property for general maps, in terms of the cotangent space of $\vf$; see Section~\ref{se:local-algebra} for details. 

The  more general  hypothesis  in our theorem  as opposed to \cite[Theorem 1.1]{Brochard:2017}, gives us the freedom to reduce its proof to the surjective case, using standard results in commutative algebra. This  is instrumental in leading to a proof that is simpler than  the one presented in \cite{Brochard:2017} as we can  apply an induction on $\edim A - \edim B$. A key ingredient in this induction step is a theorem of Nagata that tracks the projective dimension of finitely generated modules along surjective exceptional complete intersection maps; see \ref{ch:Nagata}. In fact, our main theorem, \ref{th:Brochard}, can be seen as a converse to this result. A ``better" converse is proved in Theorem~\ref{th:Brochard2}.

Work of  Calegari and Geraghty \cite{Calegari/Geraghty:2018} in ``positive defect’’ situations,   extending the patching method  in  \cite{Taylor/Wiles:1995} to situations where one patches complexes rather than modules, suggests a version of Theorem~\ref{th:intro} dealing with complexes;  see also  \cite{Khare/Thorne:2017, Thorne:2016}.    The Calegari-Geraghty patching argument leads to a situation where one can apply Proposition \ref{pr:AB} to prove the faithfulness of the action of a deformation ring on the  cohomology in ``top degree’’ of a certain arithmetic manifold. We expect that a version  of Theorem \ref{th:intro}  for complexes would allow one in principle to obviate the  need to patch in positive defect situations, just like the application of  \cite[Theorem 1.1]{Brochard:2017}  in \cite[\S 4]{Newton/Thorne:2021}. We hope to return to this topic in the future.
 
\section{Local algebra}
\label{se:local-algebra}
This section is mostly a recollection of some basic facts from local algebra, for use in the proofs in Section~\ref{se:main}.  The only new material concerns a  class of local maps introduced here called exceptional complete intersections.  Much of the discussion in this section is geared towards characterizing these maps in terms of their tangent spaces. This is easy to do for surjective maps, and even for maps essentially of finite type. With an eye towards future applications we treat more general maps of noetherian rings. To that end, it will be convenient to use Andr\'e-Quillen homology modules.  In fact, we require only the components in degrees zero and one, which are easy to describe, as we do below, and the Jacobi-Zariski sequence. Most of what is needed is already in \cite{Lichtenbaum/Schlessinger:1967}; see also \cite{Avramov/Iyengar:2002, Iyengar:2007}. We also need to refer to \cite{Avramov:1999} for more recent work on complete intersection maps. 

\begin{chunk}
\label{ch:aq-basic}
Let  $A\to B$ be a map of rings and $M$ a $B$-module. We write $\aqh iBAM$ for the $i$th Andr\'e-Quillen homology of the $A$-algebra $B$, with coefficients in $M$, namely:
\[
\aqh iBAM \colonequals \hh_i(\mathrm{L}_{B/A}\otimes_B M)
\]
where $\mathrm{L}_{B/A}$ is the cotagent complex of the map $A\to B$. One has $\aqh iBAM=0$ for $i<0$ and
\[
\aqh 0BAM = \Omega_{B/A}\otimes_BM
\]
where $\Omega_{B/A}$ is the module of K\"ahler differentials of $B$ over $A$. If $A\twoheadrightarrow B$ is surjective, with kernel $I$, then $\aqh 0BAM=0$ and from \cite[Lemma~3.1.2]{Lichtenbaum/Schlessinger:1967} one gets that
\begin{equation}
\label{eq:aq1}
\aqh 1BAM = \frac I{I^2}\otimes_B M\,.
\end{equation}
Given maps of rings $A\to B\to C$ and a $C$-module $N$ one has an exact sequence
\[
\cdot \lra \aqh {i+1}CBN \lra \aqh {i}BAN \lra  \aqh {i}CAN \lra \aqh {i}CBN \lra  \cdots 
\]
called the  Jacobi-Zariski exact sequence associated to the maps. We only need the part of the sequence for $i\le 1$, and for this see~\cite[2.3.5]{Lichtenbaum/Schlessinger:1967}.
\end{chunk}
  
 \begin{chunk}
 \label{ch:local-ring}
 By a \emph{local ring} we mean a commutative, noetherian, local ring. We write $\fm_A$ for the maximal ideal of $A$ and $k_A$ for the residue field $A/\fm_A$. The \emph{cotangent} space of $A$ is the $k_A$-vector space
 \[
 \frac{\fm_A}{\fm^2_A} \cong \aqh 1{k_A}A{k_A}\,;
 \]
the isomorphism is by \eqref{eq:aq1}. The rank of this vector space is the \emph{embedding dimension} of $A$, denoted $\edim A$. There is an inequality $\dim A\le \edim A$, where $\dim A$ is the Krull dimension of $A$; when equality holds, the ring $A$ is said to be \emph{regular}. In this case one has $\aqh i{k_A}A-=0$ for $i\ge 2$, and the converse holds as well; see \ref{ch:ci-surjective}, keeping in mind that $A$ is regular if and only if the quotient map $A\to k_A$ is complete intersection. 
\end{chunk}

\begin{chunk}
\label{ch:local-map}  
We say $\vf\colon A\to B$ is a \emph{local map} to mean that $A$ and $B$ are local rings and $\vf$ is a homomorphism  of rings satisfying $\vf(\fm_A)\subseteq \fm_B$. The \emph{closed fiber} of such a map is the local ring $B/\fm_AB$, with maximal ideal $\fm_B/\fm_AB$, and residue field $k_B$.
 
It is helpful to think of $\aqh 1BA{k_B}$ as the \emph{cotangent space} of the map $A\to B$; when $B=A/I$ it is $I/\fm_A I$.  Here is a simple but useful observation. 
 
 \begin{lemma}
 \label{le:edim}
Let $A\to B$ be a surjective local map. The rank of the map
 \[
 \aqh 1BA{k_B} \lra \aqh 1{k_B}A{k_B}
 \]
 induced by the surjection $B\to k_B$  equals $\edim A - \edim B$.
 \end{lemma}
 
 \begin{proof}
 Set $k\colonequals k_B$. The Jacobi-Zariski sequence induced by $A\to B\to k$ starts as
 \[
 \aqh 1BAk \lra \aqh 1kAk \lra \aqh 1kBk \lra 0
 \]
 The result is clear from the additivity of rank on exact sequences.
 \end{proof}

\end{chunk}

\begin{chunk}
\label{ch:regular-maps}
Following \cite{Avramov/Foxby/Herzog:1994}, we say that a local map $\vf\colon A\to A'$ is \emph{weakly regular} if it is flat and its closed fiber $A'/\fm_AA'$ is regular. In this case there are equalities
\begin{align*}
\dim  A' &= \dim A + \dim (A'/\fm_AA') \\
\edim  A' &= \edim A + \dim (A'/\fm_AA') 
\end{align*}
The first equality is well-known; see \cite[Theorem~A.11]{Bruns/Herzog:1998}.  One way to verify the second one is to consider the maps $A'\to F\to k'$,  where $F\colonequals A'/\fm_A A'$ and $k'\colonequals k_{A'}$, and the associated Jacobi-Zariski sequence
\[
0=\aqh 2{k'}F{k'} \lra \aqh 1F{A'}{k'} \lra \aqh 1{k'}{A'}{k'} \lra \aqh 1{k'}F{k'}\lra 0
\]
where the equality holds because $F$ is a regular local ring; see \ref{ch:local-ring}. Since $\vf$ is flat, base change~\cite[2.3.2]{Lichtenbaum/Schlessinger:1967}  yields
\[
\aqh 1F{A'}{k'}\cong \aqh 1{k_A}A{k'} \cong \aqh 1{k_A}A{k_A}\otimes_{k_A}k'\,.
\]
Thus the sequence above gives the desired equality.
\end{chunk}

\begin{chunk}
Given an $A$-module $M$ and a nonnegative integer $n$, we write $\fdim_AM\le n$  to indicate that $M$ has flat dimension $\le n$; that is to say,  $M$ has a flat resolution of length at most $n$. This condition is equivalent to $\Tor^A_{n+1}(-,M)=0$. When these conditions hold, $\Tor^A_{i}(-,M)=0$ for all $i\ge n+1$.

When the $A$-module $M$ is finitely generated, its flat dimension coincides with its projective dimension.
\end{chunk}

\begin{chunk}
\label{ch:fdim-regular-maps}
If $\vf\colon A\to A'$ is weakly regular, then for any finitely generated $A'$-module $N$ there are inequalities
\begin{align*}
\pdim_{A'}N 
	&\le \fdim_{A}N + \edim(A'/\fm_AA') \\
	&= \fdim_AN + \edim A' - \edim A\,.
\end{align*}
See \cite[Lemma~3.2]{Avramov/Foxby/Herzog:1994} for the inequality; the equality is from \ref{ch:regular-maps}.
\end{chunk}

\begin{chunk}
\label{ch:regular-factorization}
A \emph{regular factorization} of a local map $\vf\colon A\to B$ is a decomposition 
\[
A\xra{\ \dot\vf\ }  A' \xra{\ \vf'\ } B
\]
of $\vf$ where $\dot\vf$ is weakly regular and $\vf'$ is surjective; see  \cite{Avramov/Foxby/Herzog:1994}. It is clear that such factorizations exist if $\vf$ is finite and, more generally, essentially of finite type, that is to say, when $B$ is a localization of a finitely generated $A$-algebra. The main result of \cite{Avramov/Foxby/Herzog:1994} is that such factorizations exist also when $B$ is complete with respect to its $\fm_B$-adic topology; this extends Cohen's structure theorem for complete local rings.

The map $\vf$ is \emph{formally smoothable} if there exists a factorization as above of the composition $A\to B\to \widehat{B}$, where $\widehat{B}$ is the $\fm_B$-adic completion of $B$, in which the closed fiber of $\dot\vf$ is even geometrically regular over $k_A$, the residue field of $A$. Such factorizations exist when, for example, $\vf$ is essentially of finite type, or the extension of fields $k_A\to k_{B}$ is separable; see \cite[(1.1.2)]{Avramov/Foxby/Herzog:1994}.
\end{chunk}

\begin{chunk}
\label{ch:ci-surjective}
A surjective local map $A\to B$ is \emph{complete intersection} if its kernel can be generated by a regular sequence. In this case, if $B=A/I$, then $I/I^2$ is a free $B$-module, of rank equal to $\dim A-\dim B$. Thus
\begin{equation}
\label{eq:ci-con}
\aqh  1BAM \cong \frac I{I^2}\otimes_BM \cong M^{c} \quad \text{where $c=\dim A-\dim B$.}
\end{equation}
Moreover $\aqh iBA-=0$ for $i\ge 2$, and  this property characterizes the complete intersection property~\cite[(1.2)]{Avramov:1999}.
\end{chunk}

\begin{chunk}
\label{ch:ci-maps}
Let $\vf\colon A\to B$ be a local map and let $\widehat{B}$ be the $\fm_B$-adic completion of $B$.  Following \cite{Avramov:1999}, we say that $\vf$ is \emph{complete intersection} if in some regular factorization 
\[
A\to A'\xra{\vf'} \widehat{B}
\]
of the composed map $A\to B\to \widehat{B}$, the surjective map $\vf'$ is complete intersection in the sense of \ref{ch:ci-surjective}. This property is independent of the choice of factorization. When $A\to A''\to B$ is a regular factorization of $\vf$, $\vf$ is complete intersection if and only if $A''\to B$ is complete intersection. The class of complete intersections is closed under composition and flat base change; see~\cite{Avramov:1999} for proofs of these assertions.

Using regular factorizations of maps one can often reduce questions about general complete intersection maps to the surjective case; see, in particular, the proof of Theorem~\ref{th:Brochard}.
\end{chunk}

\begin{lemma}
\label{le:eci}
When $\vf$ is complete intersection there is an inequality
\[
\edim A - \dim A  \le  \edim B - \dim B\,.
\]
\end{lemma}

\begin{proof}
We may assume $B$ is $\fm_B$-adically complete. Consider a regular factorization $A\to A'\to B$ of $\vf$. Since $\vf$ is complete intersection, Lemma~\ref{le:edim} and \eqref{eq:ci-con} yield
\[
\edim A' - \edim B \le \dim A' - \dim B\,.
\]
This gives the inequality below:
\[
\edim A - \dim A = \edim A' -   \dim A' \leq \edim B - \dim B
\]
The equality is from \ref{ch:regular-maps}.
\end{proof}

\begin{chunk}
\label{ch:eci}
We say that a complete intersection map $\vf$ is \emph{exceptional} if the inequality in Lemma~\ref{le:eci} is an equality: $\vf$ is complete intersection and 
\[
\edim A - \dim A = \edim B - \dim B\,.
\]
Thus, for example, for such a $\vf$, the ring $A$ is regular if and only if $B$ is regular. Moreover, one can check easily using a regular factorization of $\vf$ that one also has
\[
\dim A - \depth A = \dim B - \depth B\,.
\]
This means that $A$ is Cohen-Macaulay if and only if $B$ is.

It is immediate from \ref{ch:regular-maps} that weakly regular maps are exceptional complete intersections. Here is another simple family of examples with this property.

\begin{example}
\label{ex:rlrs}
Let $A,B$ be regular local rings. Then any local map $\vf\colon A\to B$ is exceptional complete intersection. Indeed, since
\[
\edim A -\dim A = 0 = \edim B - \dim B
\]
it suffices to verify that $\vf$ is complete intersection. This can be checked using Andr\'e-Quillen homology. Here is a direct argument: any surjective map between regular rings is complete intersection, by Chevalley's theorem~\cite[Proposition~2.2.4]{Bruns/Herzog:1998}, and one can reduce to this case using regular factorizations~\ref{ch:regular-factorization}.
\end{example}

The exceptional complete intersection property can be expressed in terms of Andr\'e-Quillen homology, at least for formally smoothable maps.

\begin{lemma}
\label{le:eci-tangents}
Let $\vf\colon A\to B$ be a formally smoothable local map that is complete intersection. Then $\vf$ is exceptional if and only if the natural map
\[
\aqh 1B{A}{k_B} \lra \aqh 1{k_A}{A}{k_A}\otimes_{k_A}k_B
\]
is one-to-one.
\end{lemma}

\begin{proof}
In what follows we write $k$ for $k_B$. First consider the case when $\vf$ is surjective and complete intersection; say $B=A/I$. The complete intersection property implies that the rank of the $k$-vector space $\aqh 1BAk$ is $\dim A-\dim B$, by \eqref{eq:ci-con}. Given Lemma~\ref{le:edim} it is then clear that $\vf$ is also exceptional if and only if the map 
\[
\aqh 1BAk \to \aqh 1kAk
\]
is one-to-one.
 
In  the general case, we can assume $B$ is $\fm_B$-adically complete.  Let $A\to A'\to B$ be a regular factorization of $\vf$. Since $\aqh 1{A'}A-=0$, by \cite[(1.1)]{Avramov:1999}, for any $A'$-algebra $C$, it follows from the Jacobi-Zariski sequence  induced by $A\to A'\to C$ that the natural map is one-to-one:
\[
\aqh 1CA- \lra \aqh 1C{A'}-
\]
This justifies the exactness in the rows in the diagram below:
\[
\begin{tikzcd}[column sep=small]
0 \arrow[r] & \aqh 1BAk \arrow[r]\arrow[d] & \aqh 1B{A'}k \arrow[r]\arrow[d] 
		&\aqh 0{A'}Ak \arrow[r] \arrow[d, equal] & \cdots \\
0 \arrow[r] &  \aqh 1kAk  \arrow[r] & \aqh 1k{A'}k \arrow[r] &\aqh 0{A'}Ak \arrow[r] & \cdots
\end{tikzcd}
\]
The diagram is commutative by functoriality of Andr\'e-Quillen homology. The vertical maps on the left and in the middle are induced by $B\to k$.  A simple diagram chase reveals that if one of these is injective then so is the other. It remains to observe that since $A'\to B$ is surjective, the already established case of the result implies the middle one is injective if and only if $A'\to B$ is exceptional.
\end{proof}
\end{chunk}

It follows  from Lemma~\ref{le:eci-tangents} when $\vf\colon A\to B$ is surjective, it is exceptional complete intersection if and only if $\Ker(\vf)$ is generated by a regular sequence whose image in $\fm_A/\fm_A^2$ is a linearly independent set. The forward implication need not hold when $\vf$ is not surjective. 

\begin{example}
\label{ex:finally}
Let $k$ be a field and consider the map of $k$-algebras
\[
\vf\colon k[\!|x,y,z|\!] \lra k[\!|s,t|\!] \quad\text{where $x\mapsto s^2$, $y\mapsto st$, $z\mapsto t^2$.}
\]
This map is exceptional complete intersection; see Example~\ref{ex:rlrs}.  It is easy to check that $\Ker(\vf) = (xz - y^2)$; in particular, it is contained in $\fm_A^2$.
\end{example}

Next we record a result of Nagata~\cite[\S27]{Nagata:1975}  that tracks the projective dimension of modules along surjective exceptional complete intersections.

\begin{chunk}
\label{ch:Nagata}
Let $\vf\colon A\to B$ be a surjective local map and $N$ a finitely generated $B$-module. If $\vf$ is exceptional complete intersection, then  
\[
\edim A - \pdim_A N  =   \edim B -  \pdim_B N\,.
\]
In particular, $\pdim_AN$ is finite if and only if $\pdim_BN$ is finite.
\end{chunk}

See Theorem~\ref{th:Nagata} for a version dealing with non-necessarily surjective exceptional complete intersections, and
Theorems~\ref{th:Brochard} and \ref{th:Brochard2} for  converses.

\section{Criteria for detecting exceptional complete intersections}
\label{se:main}

We begin this section by proving Theorem~\ref{th:intro} from the Introduction, in a slightly more elaborate version. The last part recovers \cite[Theorem~1.1]{Brochard:2017}. We should note that if the $B$-module $N$ happens to be finitely generated over $A$, as would be the case if $\vf$ is a finite map, then its flat dimension equals the projective dimension.

\begin{theorem}
\label{th:Brochard}
Let $\vf\colon A\to B$ be a local map with $\edim A\ge \edim B$. If there exists a nonzero finitely generated $B$-module $N$ satisfying $\fdim_AN\le \edim A-\edim B$, then the following conclusions hold:
\begin{enumerate}[\quad\rm(1)]
\item
$N$ is free as a $B$-module;
\item 
$\vf$ is an exceptional complete intersection;
\item
$\fdim_AN = \fdim_AB = \edim A - \edim B$.
\end{enumerate}
In particular, if  $N$ is flat as an $A$-module, then $\edim A=\edim B$ and $\vf$ is flat.
\end{theorem}

\begin{proof}
We first reduce to the case when $\vf$ is surjective. To that end, note that if $\widehat B$ is the completion of $B$ at its maximal ideal, then for any finitely generated $B$-module $W$ the flatness of the map $B\to \widehat B$ implies
\[
\Tor^A_i(-,W\otimes_B\widehat B) \cong \Tor^A_i(-,W)\otimes_B \widehat B \quad\text{for each $i$.}
\]
The completion map is  faithful, so $\fdim_A(W\otimes_B\widehat B)=\fdim_AW$. Also, the $B$-module $W$ is free if and only if the $\widehat B$-module $W\otimes_B\widehat B$ is free. So replacing $B$ by $\widehat B$ we can assume $\vf$ has a regular factorization:
\[
A\xra{\ \dot\vf\ } A' \xra{\ \vf'\ } B\,;
\]
see \ref{ch:regular-factorization}. The first inequality below is by \ref{ch:fdim-regular-maps}:
\begin{align*}
\pdim_{A'} N
	& \le \fdim_AN + \dim (A'/\fm_AA')\\
        &\le \edim A - \edim B + \dim (A'/\fm_AA')\\
        &= \edim A' - \edim B
\end{align*}
The second one is our hypothesis, whilst the equality is from  \ref{ch:regular-maps}. We claim that it suffices to prove the result for $\vf'$.

Indeed, this is clear for (1) as it concerns only the $B$-module structure on $N$. For  (2) the desired conclusion holds because $\vf$ is complete intersection if and only if $\vf'$ is, by \ref{ch:ci-maps}, and then the that fact that one of them is  exceptional if and only if the other is is immediate from \ref{ch:regular-maps}. Finally if the equality in (3) holds for $\vf'$ then the inequalities above become equalities, so the desired equality holds also for $\vf$.

Thus we assume $\vf$ is surjective and hence that $N$ is finitely generated over $A$.

We first settle the case when the $B$-module $N$ is faithful, that is to say, when $\ann_AN = \Ker(\vf)$. We argue by induction on $\edim A-\edim B$. The base case is when $\edim A=\edim B$, in which case $\fdim_AN\le 0$, that is to say, $N$ is free as an $A$-module. This implies $\Ker(\vf)=0$ and then the desired result is clear. 

Suppose that $\edim A - \edim B\ge 1$. Then the ideal $\Ker(\vf)$ is nonzero, and so the hypothesis that $\pdim_AN$ is finite implies $\ann_AN$, which is $\Ker(\vf)$, contains a nonzero divisor, by a result of Auslander and Buchsbaum~\cite[Lemma~6.1]{Auslander/Buchsbaum:1958}. Moreover since $\edim A - \edim B\ge 1$, the ideal $\Ker(\vf)$ is not contained in $\fm_A^2$, so there exists a nonzero divisor, say $x$, in $\Ker(\vf)\setminus \fm_A^2$. The existence of such an $x$ follows by a standard ``prime avoidance argument" where at most two of the ideals are allowed not to be prime; see, for example, \cite[Theorem~81]{Kaplansky:1974}. Consider the factorization
\[
A\lra A/xA\xra{\ \overline{\vf}\ } B
\]
of $\vf$. Since $x$ is a nonzero divisor not contained in $\fm_A^2$ there is an equality
\[
\pdim_{A/xA} N = \pdim_A N - 1\,;
\]
see \ref{ch:Nagata}. Evidently $\edim(A/xA) = \edim A -1$. Thus the induction hypothesis applies to $\overline{\vf}$, and yields that $\overline{\vf}$ is an exceptional complete intersection, and also that $N$ is free as a $B$-module. It remains to note that, by the choice of $x$, the map $\vf$ is also an exceptional complete intersection. 

This completes the proof when $N$ is faithful as a $B$-module.

For the general case, set $I\colonequals \ann_BN$ and $\overline{B} \colonequals B/I$. Consider the composition
\[
\psi\colon A\xra{\ \vf\ } B\lra \overline{B} \,.
\]
Since $N$ viewed as a $\overline{B}$-module is faithful and  $\edim \overline{B}\le \edim B$, the already verified case of the result applied to  $\psi$ yields that $N$ is free as a $\overline{B} $-module and also that $\psi$ is an exceptional complete intersection. To complete the proof, it suffices to verify that $I=0$.

As $\psi$ is an exceptional complete intersection, one gets the first equality below.
\begin{align*}
\edim A - \edim \overline{B}  
	&= \pdim_A\overline{B} \\
	& =  \pdim_AN  \\
	& \le \edim A - \edim B \\
	& \le \edim A - \edim \overline{B} \,.
\end{align*}
The second equality holds because $N$ is free as a $\overline{B}$-module, the first inequality is by hypothesis, whilst the second one holds because $\overline{B}$ is a quotient of $B$. It follows that $\edim B = \edim \overline{B} $, that is to say, $I\subseteq \fm_B^2$. The maps $A\to B\to \overline{B}$ induce the exact row in the diagram below
\[
\begin{tikzcd}
\aqh 1BAk \arrow[dr] \arrow[r,"\pi"] & \aqh 1{\overline{B}}{A}k \arrow[d,hookrightarrow]  \arrow[r] & \aqh 1{\overline{B}}{B}k \arrow[r] & 0 \\
&\aqh 1kAk
\end{tikzcd}
\]
The diagonal arrow is induced by the maps $A\to B\to k$ and the vertical one by the maps $A\to \overline{B}\to k$; it is an inclusion because $A\to \overline B$ is an exceptional complete intersection; see Lemma~\ref{le:eci-tangents}. Since $\edim B = \edim\overline{B}$, the rank of the latter two  maps are equal, by Lemma~\ref{le:edim}. It follows that $\pi$ is onto and hence 
\[
\frac I{\fm_B I} =  \aqh 1{\overline{B}}Bk=0\,.
\]
Nakayama's lemma yields $I=0$, as desired. 
\end{proof}

Here is a variation on Theorem~\ref{th:Brochard}; unlike in that statement, we have to assume a priori that the projective dimension of the $B$-module $N$ is finite. Compare the hypotheses with Nagata's Theorem~\ref{ch:Nagata}. 

\begin{theorem}
\label{th:Brochard2}
Let $\vf\colon A\to B$ be a local map and $N$ a nonzero finitely generated $B$-module $N$ of finite projective dimension. There is then an inequality
\[
\fdim_AN - \pdim_BN \ge \edim A-\edim B\,.
\]
Moreover, if equality holds then $\vf$ is an exceptional complete intersection. 
\end{theorem}

The inequality in the statement  can be strict: let $A$ be a field, $B$ a ring of formal power series over $A$,  and set $N\colonequals B$.  

\begin{proof}
We can assume $\fdim_AN$ is finite. First we treat the case when $\vf$ is surjective.  Since $N$ is a nonzero finitely generated $B$-module that is of finite projective dimension over both $A$ and $B$, it follows that $\pdim_AB$ is finite, by \cite[Theorem IV]{Foxby/Iyengar:2001}; see also \cite[Remark~5.6]{Dwyer/Greenlees/Iyengar:2006}. This fact will be crucial in the ensuing proof.

Keep in mind that the projective dimension of $N$ over $B$, and that of $N$ and $B$ over $A$ are finite. Thus, repeatedly applying the equality of Auslander and Buchsbaum~\cite[Theorem~1.3.3]{Bruns/Herzog:1998} yields
\begin{align*}
\pdim_AN - \pdim_BN 
	&=(\depth A - \depth_AN) - (\depth B - \depth_BN) \\
	&=\depth A - \depth B\\
	&=\pdim_AB
\end{align*}
We have also used the fact that the depth of $M$ over $A$ is the same as that over $B$. It thus suffices to prove the result for $N=B$, namely that there is an inequality
\[
\pdim_AB \ge \edim A - \edim B
\]
The fact that $\varphi$ is an exceptional complete intersection if equality holds then follows from Theorem~\ref{th:Brochard}.

From this point on the proof flows as in that of Theorem~\ref{th:Brochard}: We argue by induction on $\edim A - \edim B$; the base case when this number is zero is again a tautology. When $\edim A - \edim B\ge 1$ one can find a nonzero divisor $x$ in $\Ker(\vf) \setminus \fm_A^2$. Then from Nagata's theorem~\ref{ch:Nagata} one gets the equality below
\begin{align*}
\pdim_A B 
	&= \pdim_{A/xA}B +1 \\
	&  \ge \edim (A/xA) -\edim B + 1 \\
        &= \edim A - \edim B
\end{align*}
The inequality is by the induction hypothesis. This gives the desired inequality.

This completes the proof of the result when $\vf$ is surjective. The general case is settled by the standard reduction to the surjective one: Completing $B$ at its $\fm_B$-adic topology, we can assume $\vf$ has a regular factorization $A\to A'\to B$. Then one has inequalities, where \ref{ch:fdim-regular-maps} gives the first one 
\begin{align*}
\fdim_AN - \pdim_BN 
	&\ge \pdim_{A'}N  -\pdim_B N + \edim A - \edim A' \\
	&\ge \edim A - \edim B
\end{align*}
and the second one is by the already established case of the result, applied to the surjection $A'\to B$. This gives the stated inequality, and also that if equality holds, then $\pdim_{A'}N - \pdim_BN = \edim A' - \edim B$, which implies $\vf'$ is exceptional complete intersection; thus $\vf$ has this property, by definition.
\end{proof}

The result below complements Theorem~\ref{th:Brochard2} and is an extension of Nagata's Theorem~\ref{ch:Nagata} to  maps that may not be surjective.

\begin{theorem}
\label{th:Nagata}
If  a local map $\vf\colon A\to B$ is exceptional complete intersection, then for any finitely generated $B$-module $N$ one has
\[
\edim B - \pdim_BN \ge \edim A - \fdim_AN \,;
\]
in particular, $\fdim_AN$ is finite if and only if $\pdim_BN$ is finite. Equality holds when the map $\vf$ is finite.
\end{theorem}

The inequality in the statement can be strict if $\vf$ is not finite: Let $k$ be a field and $\vf\colon k[\![x,y]\!]\to k[\![s]\!]$ the map of $k$-algebras that maps $x,y$ to $0$, and take $N=k[\![s]\!]$.

\begin{proof}
As before we can assume $\vf$ has a regular factorization $A\to A'\to B$. Then it follows from \ref{ch:fdim-regular-maps} that $\fdim_AN$ is finite if and only $\pdim_{A'}N$ is finite; since $A'\to B$ is a surjective exceptional complete intersection, Nagata's Theorem~\ref{ch:Nagata} yields that the latter holds if and only if $\pdim_BN$ is finite.  Thus in the rest of the argument we can assume that $\fdim_AN$ and $\pdim_BN$ are finite. Then the desired inequality is already in Theorem~\ref{th:Brochard2}. 

When $\vf$ is finite,  $N$ is also finitely generated over $A$, so $\fdim_AN = \pdim_AN$. This and the Auslander-Buchsbaum equality give the first equality below
\begin{align*}
\edim A - \fdim_A N 
	&= \edim A - \depth A + \depth_A N \\
	&=\edim B - \depth B + \depth_B N \\
	&=\edim B - \pdim_BN \,.
\end{align*}
The second equality follows from the displayed equalities in \ref{ch:eci}, and the fact that depth of $N$ over $A$ equals its depth over $B$, since $\vf$ is finite. The last equality is again by  the Auslander-Buchsbaum equality.
\end{proof}


\begin{bibdiv}
\begin{biblist}

\bib{Auslander/Buchsbaum:1958}{article}{
   author={Auslander, Maurice},
   author={Buchsbaum, David A.},
   title={Codimension and multiplicity},
   journal={Ann. of Math. (2)},
   volume={68},
   date={1958},
   pages={625--657},
   issn={0003-486X},
   review={\MR{99978}},
   doi={10.2307/1970159},
}

\bib{Avramov:1999}{article}{
   author={Avramov, Luchezar L.},
   title={Locally complete intersection homomorphisms and a conjecture of
   Quillen on the vanishing of cotangent homology},
   journal={Ann. of Math. (2)},
   volume={150},
   date={1999},
   number={2},
   pages={455--487},
   issn={0003-486X},
   review={\MR{1726700}},
   doi={10.2307/121087},
}
\bib{Avramov/Foxby/Herzog:1994}{article}{
   author={Avramov, Luchezar L.},
   author={Foxby, Hans-Bj\o rn},
   author={Herzog, Bernd},
   title={Structure of local homomorphisms},
   journal={J. Algebra},
   volume={164},
   date={1994},
   number={1},
   pages={124--145},
   issn={0021-8693},
   review={\MR{1268330}},
   doi={10.1006/jabr.1994.1057},
}

\bib{Avramov/Iyengar:2002}{article}{
   author={Avramov, Luchezar L.},
   author={Iyengar, Srikanth},
   title={Homological criteria for regular homomorphisms and for locally
   complete intersection homomorphisms},
   conference={
      title={Algebra, arithmetic and geometry, Part I, II},
      address={Mumbai},
      date={2000},
   },
   book={
      series={Tata Inst. Fund. Res. Stud. Math.},
      volume={16},
      publisher={Tata Inst. Fund. Res., Bombay},
   },
   date={2002},
   pages={97--122},
   review={\MR{1940664}},
}
\bib{Brochard:2017}{article}{
   author={Brochard, Sylvain},
   title={Proof of de Smit's conjecture: a freeness criterion},
   journal={Compos. Math.},
   volume={153},
   date={2017},
   number={11},
   pages={2310--2317},
   issn={0010-437X},
   review={\MR{3692747}},
   doi={10.1112/S0010437X17007370},
}

\bib{Bruns/Herzog:1998}{book}{
   author={Bruns, Winfried},
   author={Herzog, J\"{u}rgen},
   title={Cohen-Macaulay rings (revised edition)},
   series={Cambridge Studies in Advanced Mathematics},
   volume={39},
   publisher={Cambridge University Press, Cambridge},
   date={1998},
   pages={xiv+453},
   isbn={0-521-41068-1},
   review={\MR{1251956}},
}


\bib{Dwyer/Greenlees/Iyengar:2006}{article}{
   author={Dwyer, W.},
   author={Greenlees, J. P. C.},
   author={Iyengar, S.},
   title={Finiteness in derived categories of local rings},
   journal={Comment. Math. Helv.},
   volume={81},
   date={2006},
   number={2},
   pages={383--432},
   issn={0010-2571},
   review={\MR{2225632}},
   doi={10.4171/CMH/56},
}


\bib{Calegari/Geraghty:2018}{article}{
   author={Calegari, Frank},
   author={Geraghty, David},
   title={Modularity lifting beyond the Taylor-Wiles method},
   journal={Invent. Math.},
   volume={211},
   date={2018},
   number={1},
   pages={297--433},
   issn={0020-9910},
   review={\MR{3742760}},
   doi={10.1007/s00222-017-0749-x},
}


\bib{Diamond:1997}{article}{
   author={Diamond, Fred},
   title={The Taylor-Wiles construction and multiplicity one},
   journal={Invent. Math.},
   volume={128},
   date={1997},
   number={2},
   pages={379--391},
   issn={0020-9910},
   review={\MR{1440309}},
   doi={10.1007/s002220050144},
}

\bib{Foxby/Iyengar:2001}{article}{
   author={Foxby, Hans-Bj\o rn},
   author={Iyengar, Srikanth},
   title={Depth and amplitude for unbounded complexes},
   conference={
      title={Commutative algebra},
      address={Grenoble/Lyon},
      date={2001},
   },
   book={
      series={Contemp. Math.},
      volume={331},
      publisher={Amer. Math. Soc., Providence, RI},
   },
   date={2003},
   pages={119--137},
   review={\MR{2013162}},
   doi={10.1090/conm/331/05906},
}

\bib{Iyengar:2007}{article}{
   author={Iyengar, Srikanth},
   title={Andr\'{e}-Quillen homology of commutative algebras},
   conference={
      title={Interactions between homotopy theory and algebra},
   },
   book={
      series={Contemp. Math.},
      volume={436},
      publisher={Amer. Math. Soc., Providence, RI},
   },
   date={2007},
   pages={203--234},
   review={\MR{2355775}},
   doi={10.1090/conm/436/08410},
}

\bib{Kaplansky:1974}{book}{
   author={Kaplansky, Irving},
   title={Commutative rings},
   edition={Revised edition},
   publisher={The University of Chicago Press, Chicago, Ill.-London},
   date={1974},
   pages={ix+182},
   review={\MR{0345945}},
}

\bib{Khare/Thorne:2017}{article}{
   author={Khare, Chandrashekhar B.},
   author={Thorne, Jack A.},
   title={Potential automorphy and the Leopoldt conjecture},
   journal={Amer. J. Math.},
   volume={139},
   date={2017},
   number={5},
   pages={1205--1273},
   issn={0002-9327},
   review={\MR{3702498}},
   doi={10.1353/ajm.2017.0030},
}

\bib{Lichtenbaum/Schlessinger:1967}{article}{
   author={Lichtenbaum, S.},
   author={Schlessinger, M.},
   title={The cotangent complex of a morphism},
   journal={Trans. Amer. Math. Soc.},
   volume={128},
   date={1967},
   pages={41--70},
   issn={0002-9947},
   review={\MR{209339}},
   doi={10.2307/1994516},
}

\bib{Nagata:1975}{book}{
   author={Nagata, Masayoshi},
   title={Local rings},
   note={Corrected reprint},
   publisher={Robert E. Krieger Publishing Co., Huntington, N.Y.},
   date={1975},
   pages={xiii+234},
   isbn={0-88275-228-6},
   review={\MR{0460307}},
}

\bib{Newton/Thorne:2021}{article}{
author={Newton, James},
author={Thorne, Jack A.},
title={Adjoint Selmer groups of automorphic Galois representations of unitary type},
journal={J. Eur. Math. Soc.}, 
status={to appear},
eprint={https://arxiv.org/pdf/1912.11265.pdf},
}

\bib{Newton/Thorne:2020}{article}{
author={Newton, James},
author={Thorne, Jack A.},
title={Symmetric power functoriality for holomorphic modular forms, II},
journal={Publ. Math. Inst. Hautes \'Etudes Sci. }
status={to appear},
doi={10.1007/s10240-021-00126-4},
}

\bib{Taylor/Wiles:1995}{article}{
   author={Taylor, Richard},
   author={Wiles, Andrew},
   title={Ring-theoretic properties of certain Hecke algebras},
   journal={Ann. of Math. (2)},
   volume={141},
   date={1995},
   number={3},
   pages={553--572},
   issn={0003-486X},
   review={\MR{1333036}},
   doi={10.2307/2118560},
}

\bib{Thorne:2016}{article}{
author={Thorne, Jack A.},
title={Beyond the Taylor--Wiles method},
remark={Notes for lectures at the workshop "Deformation theory, Completed Cohomology, Leopoldt Conjecture and K-theory},
status={unpublished},
eprint={https://www.dpmms.cam.ac.uk/~jat58/beyondtw.pdf},
}

\bib{Wiles:1995}{article}{
   author={Wiles, Andrew},
   title={Modular elliptic curves and Fermat's last theorem},
   journal={Ann. of Math. (2)},
   volume={141},
   date={1995},
   number={3},
   pages={443--551},
   issn={0003-486X},
   review={\MR{1333035}},
   doi={10.2307/2118559},
}
\end{biblist}
\end{bibdiv}
\end{document}